\documentclass[11pt]{amsart}

\usepackage[initials]{amsrefs}
\usepackage{amssymb,amsmath,amsfonts,latexsym,amsthm,hyperref,graphicx,paralist,color}
\usepackage[all]{xy}

\newtheorem{theorem}{Theorem}[section]
\newtheorem{proposition}[theorem]{Proposition}
\newtheorem{corollary}[theorem]{Corollary}
\newtheorem{example}[theorem]{Example}

\newtheorem{remark}[theorem]{Remark}

\newtheorem{lemma}[theorem]{Lemma}

\newtheorem{definition}[theorem]{Definition}

\numberwithin{equation}{section}

\begin{document}

\title{Peano curves on topological vector spaces}

\date{}

\author[Albuquerque]{N. Albuquerque}
\address{Departamento de Matem\'atica, \newline \indent
Universidade Federal da Para\'iba, \newline \indent
58.051-900 - Jo\~{a}o Pessoa, Brazil.}
\email{ngalbqrq@gmail.com}

\author[Bernal]{L. Bernal-Gonz\'alez}
\address{Departamento de An\'alisis Matem\'atico,\newline\indent
Facultad de Matem\'aticas, \newline\indent
Universidad de Sevilla, \newline\indent
Apdo. 1160, Avenida Reina Mercedes, \newline\indent
Sevilla, 41080, Spain.}
\email{lbernal@us.es}

\author[Pellegrino]{D. Pellegrino}
\address{Departamento de Matem\'atica, \newline \indent
Universidade Federal da Para\'iba, \newline \indent
58.051-900 - Jo\~ao Pessoa, Brazil.}
\email{dmpellegrino@gmail.com and pellegrino@pq.cnpq.br}

\author[Seoane]{J.B. Seoane-Sep\'ulveda}
\address{Departamento de An\'alisis Matem\'atico, \newline\indent
Facultad de Ciencias Matem\'{a}ticas, \newline\indent
Plaza de Ciencias 3, \newline\indent
Universidad Complutense de Madrid,\newline\indent
Madrid, 28040, Spain.}
\email{jseoane@mat.ucm.es}

\thanks{N. Albuquerque was supported by CAPES. L. Bernal-Gonz\'alez was partially supported by the Plan Andaluz de Investigaci\'on de la Junta de Andaluc\'{\i}a FQM-127 Grant P08-FQM-03543 and by MEC Grant MTM2012-34847-C02-01. D. Pellegrino and J.B. Seoane-Sep\'{u}lveda were partially supported by CNPq Grant 401735/2013-3 (PVE - Linha 2).}


\begin{abstract}
The starting point of this paper is the existence of Peano curves, that is, continuous surjections mapping the unit interval onto the unit square. From this fact one can easily construct of a continuous surjection from the real line $\mathbb{R}$ to any Euclidean space $\mathbb{R}^n$. The algebraic structure of the set of these functions (as well as extensions to spaces with higher dimensions) is analyzed from the modern point of view of lineability, and large algebras are found within the families studied. We also investigate topological vector spaces that are continuous image of the real line, providing an optimal lineability result.
\end{abstract}

\maketitle

\tableofcontents


\section{Preliminaries}

Lately, many authors have been interested in the study of the set of surjections in $\mathbb{K}^\mathbb{K}$ ($\mathbb{K} = \mathbb{R}$ or $\mathbb{C}$). From this study, many different classes of functions have been either recovered from the old literature or introduced. Some of these classes are, for instance:
\begin{inparaenum}
\item[{\em (i)}] everywhere surjective functions (ES, see \cite{AronGS}).
\item[{\em (ii)}] strongly everywhere surjective functions (SES, see \cite{PAMS2010}),
\item[{\em (iii)}] perfectly everywhere surjective functions (PES, see \cite{PAMS2010}), and
\item[{\em (iv)}] Jones functions (J, see \cite{Gamez,GamezMS,jones1942}).
\end{inparaenum}

\vskip .1cm

If $S$ and $CS$ stand, respectively, for the set of surjections and continuous surjections on $\mathbb{R}$, the above functions (when defined on $\mathbb{R}$) enjoy the following strict inclusions:
 \[
\xymatrix{
     J \ar[r] \ar[drrr]& PES \ar[r]  \ar[drr] & SES \ar[r]  \ar[dr] & ES \ar[d]\\
     CS \ar[rrr]& & & S
    }
\]
Authors have studied the previous classes of functions in depth, to the point of even finding large algebraic structures (infinite dimensional linear spaces or infinitely generated algebras) inside the previous sets of functions. However, ``most'' of these functions, although surjective, also are nowhere continuous on their domains. Thus, the natural question rises when one tries to consider continuous surjections.

\vskip .1cm

Inspired by Cantor's counterintuitive result stating that the unit interval $[0,1]$ has the same cardinality as the infinite number of points in any finite-dimensional manifold (such as the unit square), Peano constructed the (no doubt!) most famous space filling curve, also known as the {\em Peano curve} \cite{peano,sagan} (see Figure \ref{peano_curve}). Later on, the Hahn-Mazurkiewicz theorem
(see, e.g., \cite[Theorem 31.5]{willard} or \cite{HoY}) helped in characterizing the spaces that are the continuous image of curves, namely:

\begin{figure}
\centering
\includegraphics[width=0.5\textwidth]{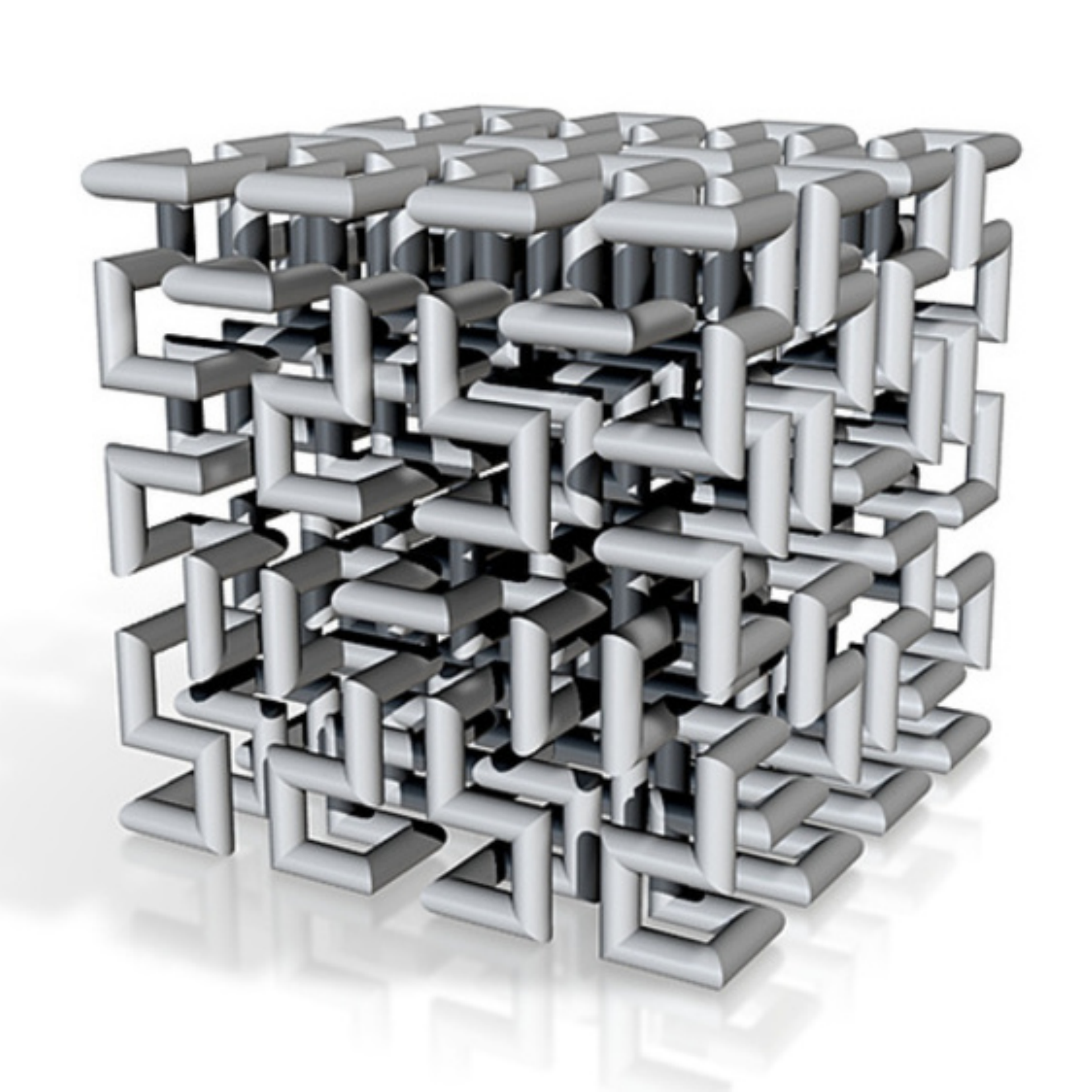}
\caption{Sketch of an iteration of a space-filling curve.}\label{peano_curve}
\end{figure}

\begin{theorem}[Hahn-Mazurkiewicz]
A non-empty Hausdorff topological space is a continuous image of the unit interval if and only if it is a compact, connected, locally connected, and second-countable space.
\end{theorem}

Hausdorff spaces that are the continuous image of the unit interval will be called {Peano spaces}, so that, a {\em Peano space} is a Hausdorff, compact, connected, locally connected second-countable topological space. Equivalently, by well-known metrization theorems, a Peano space is a compact, connected, locally connected metrizable topological space. From Peano's example one can easily construct a continuous function from the real line $\mathbb{R}$ onto the plane $\mathbb{R}^2$ (see, e.g., \cite{nga}). If $X$ and $Y$ are topological spaces, by \,$\mathcal{C} (X,Y)$ \, and \,$\mathcal{CS} (X,Y)$ \,we will denote, respectively, the set of all continuous mappings $X \to Y$ and the subfamily of all continuous {\it surjective} mappings.

\vskip .1cm

This paper focuses on studying the algebraic structure of the set of continuous surjections between Euclidean spaces, as well as extensions of Peano curves to arbitrary topological vector spaces that are, in some natural sense, sums of Peano spaces. Before carrying on, let us recall some concepts that, by now, are widely known (see, e.g., \cite{AronGS,AronSe,APS_Studia,isr,BerPS,EGS,JFA2014,LAA2014}).

\begin{definition}[lineability, \cite{AronGS,seoane2006}]
Let $X$ be a topological vector space and $M$ a subset of $X$. Let $\mu$ be a cardinal number.
\begin{enumerate}
\item[\rm (1)] $M$ is said to be {\rm $\mu$-lineable} if $M \cup \{0\}$ contains
    a vector space of dimension \(\mu\). At times, we shall be referring to the set $M$ as simply {\rm lineable} if the existing subspace is infinite dimensional.
\item[\rm (2)] When the above linear space can be chosen to be dense {\rm (}infinite dimensional and closed, resp.{\rm )} in $X$ we shall say that $M$ is {\rm $\mu$-dense-lineable} {\rm (spaceable, resp.)}.
\end{enumerate}
\end{definition}

Moreover, Bernal introduced in \cite{bernal2010} the notion of {\em maximal lineable} (and that of {\it maximal dense-lineable}) in $X$, meaning that, when keeping the above notation, the dimension of the existing linear space equals dim$(X)$. Besides asking for linear spaces one could also study other structures, such as algebrability and some related ones, which were presented in \cite{AronSe,APS_Studia,bartoszewiczglab2012b,seoane2006}.

\begin{definition} \label{def algebrable}
Given an algebra $\mathcal{A}$, a subset $\mathcal{B} \subset \mathcal{A}$, and a cardinal number $\kappa$, we say that $\mathcal{B}$ is:
\begin{enumerate}
\item[$\mathrm{(1)}$] {\rm algebrable} if there is a subalgebra $\mathcal{C}$ of $\mathcal{A}$ so
    that $\mathcal{C} \subset \mathcal{B} \cup \{ 0\}$ and the cardinality of any system of generators of \,$\mathcal{C}$ is infinite.
\item[$\mathrm{(2)}$] {\rm \(\kappa\)-algebrable} if there exists a \(\kappa\)-generated subalgebra $\mathcal{C}$ of $\mathcal{A}$ with $\mathcal{C} \subset \mathcal{B} \cup \{ 0\}$.
\item[$\mathrm{(3)}$] \emph{strongly \(\kappa\)-algebrable} if there exists a \(\kappa\)-generated free algebra \(\mathcal C\) contained in \(\mathcal{B} \cup \{0\}\).
\end{enumerate}
\end{definition}

If $\mathcal{A}$ is commutative, the last sentence means that there is a set $C \subset \mathcal{A}$ with card$(C) = \kappa$
such that, for every finite set $\{x_1, \dots ,x_N \} \subset C$ and every nonzero polynomial $P$
of $N$ variables without constant term, one has $P(x_1, \dots ,x_N) \in \mathcal{B} \setminus \{0\}$.
Of course, being strongly algebrable implies being algebrable (the converse is not true, see \cite{bartoszewiczglab2012b}).
When, in part (3) of Definition \ref{def algebrable}, one can take
as $\kappa$ the supremum of the cardinalities of all algebraically free systems in $\mathcal{A}$,
then $\mathcal{B}$ will be called {\it maximal strongly algebrable} in $\mathcal{A}$.

\vskip .1cm

This paper is arranged in five sections. Section 2 is devoted to the study of the set of continuous surjections from $\mathbb R^m$ to $\mathbb C^n$ (since the case of continuous surjections from $\mathbb R^m$ to $\mathbb R^n$, and its maximal dense-lineability and spaceability, was recently studied and solved in \cite{nga,BernalOrd}). Here we shall improve the results from \cite{nga,BernalOrd} by showing that the subset of continuous surjections from $\mathbb R^m$ to $\mathbb C^n$ such that each value \,$a \in \mathbb{C}^n$ \,is assumed {\it on an unbounded set} of \,$\mathbb{R}^m$\, is, actually, (maximal) strongly algebrable (Theorem \ref{thm_RmCn}). In order to achieve this we shall need to make use of some results and machinery from Complex Analysis, such as the order of growth of an entire function. While doing this, we also provide some new results from Complex Analysis which are of independent interest (see, e.g., Lemma \ref{lemagordo}).

Henceforth, for a given arbitrary non-empty set $\Lambda$, the linear space of complex funcions $\mathbb{C}^\Lambda$ is an algebra when endowed with pointwise product of vectors. 

\vskip .1cm

Section 3 moves on to the study of generalizations of the previous results to topological vector spaces that are, in some natural sense, covered by Peano spaces. We introduce the notion of $\sigma$-Peano space (see Definition \ref{def_sigPeano}) and use it to show (among other results) that given any topological vector space $\mathcal{X}$ that is also a $\sigma$-Peano space, then the set
\[
\left\{ f \in \mathcal{C}\left(\mathbb{R}^m, \mathcal{X} \right) : f^{-1}(\{a\}) \text{ is unbounded for every } a \in \mathcal{X} \right\}
\]
is $\mathfrak{c}$-lineable (hence maximal lineable in $ \mathcal{C}\left(\mathbb{R}^m, \mathcal{X} \right)$), where $\mathfrak{c}$ stands for the continuum (Theorem \ref{thm_Peano_lineab}). In addition, we will show how, by just starting with separable normed spaces, one can obtain $\sigma$-Peano spaces. In Section 4, we analyze Peano spaces in the framework of sequences spaces. The last section is devoted to study Peano space in real and complex function spaces.

\section{Peano curves on Euclidean spaces}

Let $\mathbb{N} := \{1,2, \dots \}$ and $\mathbb{N}_0 := \mathbb{N} \cup \{0 \}$. Along this paper, and for any topological space $X$, we will use the notation
\[
\mathcal{CS}_\infty (\mathbb{R}^m,X) :=
\left\{ f \in \mathcal{C}\left(\mathbb{R}^m,X\right) :
 f^{-1}(\{a\}) \text{ is unbounded for every } a \in X \right\}.
\]
In \cite{nga}, Albuquerque showed the following.

\begin{theorem}[Albuquerque, 2014]\label{Alb}
For every pair $m,n \in \mathbb{N}$, the set \,$\mathcal{CS} (\mathbb{R}^m,\mathbb{R}^n)$ \,is maximal lineable
in the space \,$\mathcal{C} (\mathbb{R}^m,\mathbb{R}^n)$.
\end{theorem}

Also, in \cite{BernalOrd}, Bernal and Ord\'o\~nez provided the following natural generalization of the previous result.

\begin{theorem}[Bernal and Ord\'o\~nez, 2014]\label{BerOrd}
For each pair $m,n \in \mathbb{N}$, the set \,$\mathcal{CS}_\infty (\mathbb{R}^m,\mathbb{R}^n)$ \,
is maximal dense-lineable and spaceable in \,$\mathcal{C}\left(\mathbb{R}^m,\mathbb{R}^n\right)$.
In particular, it is maximal lineable in \,$\mathcal{C}\left(\mathbb{R}^m,\mathbb{R}^n\right)$.
\end{theorem}

A natural question would be to ask about the algebrability of the set $\mathcal{CS}_\infty(\mathbb{R}^m,\mathbb{R}^n)$. Clearly, algebrability cannot be obtained in the real context, since for any $f \in \mathbb{R}^\mathbb{R}$, $f^2$ is  always non-negative. However, in the complex frame it is actually possible to obtain algebrability. Before that, let us recall some results related to the growth of an entire function (see, e.g., \cite[p.9]{boas}).

\begin{remark}\label{rem_order}
{\rm ({\it Order of an entire function and consequences}). By \,$\mathcal{H}\left(\mathbb{C}\right)$ we denote the space of all entire functions from $\mathbb{C}$ to $\mathbb{C}$. For $r>0$ and $f \in \mathcal{H}\left(\mathbb{C}\right)$, we set $M\left(f,r\right):= \max_{|z|=r} \vert f(z)\vert$.
The function $M(f,\cdot)$ increases strictly towards $+\infty$ as soon as $f$ is non-constant.
\begin{itemize}
\item[(a)] The (growth) {\it order} $\rho (f)$ of an entire function $f \in \mathcal{H}\left(\mathbb{C}\right)$ is defined as the infimum of all positive real numbers $\alpha$ with the following property: $M\left(f,r\right)<e^{r^\alpha}$ \,for all $r>r(\alpha)>0$.
    Note that $\rho (f) \in [0,+\infty ]$. Trivially, the order of a constant map is $0$. If $f$ is non-constant, we have
\[
\rho(f) =  \limsup_{r\to+\infty} \frac{\log\log M\left(f,r\right)}{\log r}.
\]
\item[(b)] If $f(z) = \sum_{n=1}^\infty a_n z^n$ is the MacLaurin series expansion of $f$ then
\[
\rho(f) = \limsup_{n\to+\infty} \frac{n\log n}{\log\left(1/|a_n|\right)}.
\]
In particular, given $\displaystyle \alpha > 0, \, f_\alpha(z):= \sum_{n=1}^\infty \frac{z^n}{n^{n/\alpha}}$ \,satisfies \,$\rho(f_\alpha)=\alpha$.
\item[(c)] For every $f \in \mathcal{H}\left(\mathbb{C}\right)$, every $N \in \mathbb{N}$ and every $\alpha \in \mathbb{C}\setminus \{0\}$,
\[
\rho\left(\alpha f^N\right) = \rho\left(f\right).
\]
\item[(d)] For every  $f,g \in \mathcal{H}\left(\mathbb{C}\right)$,
$$\rho(f\cdot g) \leq \max\{\rho(f),\rho(g)\}$$ and $$\rho(f+g) \leq \max\{\rho(f),\rho(g)\}.$$

\noindent Moreover, if $f$ and $g$ have different orders, then
\[
\rho(f+g) = \max \{\rho(f),\rho(g)\} = \rho(f\cdot g),
\]
where it is assumed $f \not\equiv 0 \not\equiv g$ for the second equality.
\item[(e)] (Corollary to Hadamard's theorem): Every non-constant entire function $f$ with $\infty > \rho(f) \notin \mathbb{N}$ is surjective (see, e.g., \cite[Corollary, p.211]{ahlfors} or \cite[Thm 9.3.10]{greene.krantz}).
\end{itemize}
}
\end{remark}

As a consequence of the previous properties, we obtain the following result (of independent interest) concerning the order of a polynomial of several variables evaluated on entire functions with different orders. First, we need to establish some notation: for a non-constant polynomial in $M$ complex variables $P \in \mathbb{C}[z_1,\dots,z_M]$, let $\mathcal{I}_P \subset \{1,\ldots,M\}$ be the set of indexes $k$ such that the variable $z_k$ explicitly appears in some monomial (with non-zero coefficient) of $P$;
that is, $\mathcal{I}_P = \{n \in \{1, \dots ,M\}: \, {\partial P \over \partial z_n} \not\equiv 0\}$.

\begin{lemma}\label{lemagordo}
Let $f_1, \dots ,f_M \in \mathcal{H}(\mathbb{C})$ such that $\rho(f_i)\neq\rho(f_j)$ whenever $i \neq j$. Then
\[
\rho \left(P \left(f_1,\dots,f_M\right)\right) = \max_{k \in \mathcal{I}_P} \rho\left(f_k\right),
\]
for all non-constant polynomials $P \in \mathbb{C}[z_1,\dots,z_M]$. Moreover, $\left(f_k\right)_{k=1}^M$ is algebraically independent and generates a free algebra.
\end{lemma}

\begin{proof}
We start by proving the result for just one entire map. Let us fix a non-constant entire function. We shall prove that $\rho\left(P(f)\right) = \rho(f)$, for any non-constant polynomial $P \in \mathbb{C}[z]$. Properties (c) and (d) of Remark \ref{rem_order} assure that $\rho(P(f)) \leq \rho(f)$. So, we just need to prove the reverse inequality. We may write $P(z) = a_m z^m + \cdots + a_1 z + a_0$, with $m>0$ and $a_m \in \mathbb{C}\setminus \{0\}$. Since
\[
\lim_{|z|\to+\infty} \left\vert \frac{P(z)}{z^m} \right\vert = |a_m| >0,
\]
we have
\[
\left\vert P(z)\right\vert > c \cdot |z|^m
\]
for $|z|$ large enough, where $c := |a_m|/2$. Plainly, we may suppose $\rho (f) > 0$. Let $\varepsilon>0$ be such that $0<\varepsilon<\rho(f)$. By the definition of order, there exist a sequence of positive radii $\left(r_n\right)_n$ going to $+\infty$ and a complex sequence $\left(z_n\right)_n$
such that $|z_n| = r_n$ and
\[
\left\vert f(z_n) \right\vert = M\left(f,r_n\right) > e^{r_n^{\rho(f)-\varepsilon}} \quad (n=1,2, \dots ).
\]
Thus, for large enough values of $n$,
\[
M\left(P(f),r_n\right)
\geq \left\vert P(f)(z_n)\right\vert
> c \cdot \left\vert f(z_n) \right\vert^m
> c \cdot e^{m r_n^{\rho(f)-\varepsilon}}.
\]
Consequently,
\begin{align*}
\rho\left(P(f)\right)
& \geq \limsup_{r \to +\infty} \frac{\log \log M\left(P(f),r\right)}{\log r} \\
& \geq \limsup_{n \to \infty} \frac{\log \log M\left(P(f),r_n\right)}{\log r_n} \\
& \geq \lim_{n \to \infty} \frac{\log \log (c \cdot e^{m r_n^{\rho(f)-\varepsilon}}) }{\log r_n} \\
& \geq \rho(f) - \varepsilon.
\end{align*}
This leads us to obtain the remaining inequality $\rho\left(P(f)\right) \geq \rho(f)$ and, therefore, conclude the proof for the case $M=1$. Next, let us deal with the general case: we may assume that the functions $f_1,\dots,f_M \in \mathcal{H}(\mathbb{C})$ satisfy $\rho(f_1)<\rho(f_2)<\dots<\rho(f_M)$. Given a non-constant polynomial $P \in \mathbb{C}[z_1,\dots,z_M]$, we just need to prove, as earlier, that
\[
\rho \left(P \left(f_1,\dots,f_M\right)\right) \geq \max_{k \in \mathcal{I}_P} \rho\left(f_k\right).
\]
Let be $N = \max_{k \in \mathcal{I}_P}$. We can write
\[
P(f_1,\dots,f_M) = \sum_{i=0}^m P_i(f_1,\dots,f_{N-1}) \cdot f_N^i, \eqno (1)
\]
with some $m>0$ and $P_m \in \mathbb{C}[z_1,\dots,z_{N-1}] \setminus \{0\}$. Let $\varepsilon>0$ such that
\[
\rho(f_{N-1}) < \rho(f_N) - 2\varepsilon < \rho(f_N) =: \rho_N.
\]
Now, parts (c) and (d) of Remark \ref{rem_order} allow us to estimate the order of each one of terms of the sum in (1):
\[
\rho \left(P_i(f_1,\dots,f_{N-1})\right) \leq \rho(f_{N-1}) < \rho_N \ \text{ for all }  i=0,\dots,m \,\,\, and
\]
\[
\rho \left(P_m(f_1,\dots,f_{N-1}) \cdot f_N\right) = \rho_N.
\]
As before, there exist a sequence of positive real numbers, $\left(r_n\right)_n$, going to $+\infty$ and complex numbers $z_n$, of modulus $r_n$, such that, for $n$ large enough, the following inequalities hold:
\[
\left\vert P_m(f_1,\dots,f_{N-1})(z_n)\right\vert \cdot \left\vert f_N (z_n) \right\vert
> e^{r_n^{\rho_N-\varepsilon}} \,\,\, \hbox{and}
\]
\[
\left\vert P_i(f_1,\dots,f_{N-1})(z_n)\right\vert < e^{r_n^{\rho_N-2\varepsilon}} \ \ \text{ for all } i=0,\dots,m.
\]
In particular,
\[
\left\vert f_N (z_n) \right\vert > e^{r_n^{\rho_N-\varepsilon} - r_n^{\rho_N-2\varepsilon}} \ \ \text{ for } n \text{ large}.
\]
Thus,
\begin{align*}
& \left\vert P(f_1,\dots,f_M)(z_n)\right\vert \geq \\
& \geq \left\vert P_m(f_1,\dots,f_{N-1})(z_n)\right\vert \cdot \left\vert f_N (z_n) \right\vert^m
 - \sum_{i=0}^{m-1} \left\vert P_i(f_1,\dots,f_{N-1})(z_n)\right\vert \cdot \left\vert f_N (z_n) \right\vert^i \\
& > e^{r_n^{\rho_N-\varepsilon}} \cdot \left\vert f_N (z_n) \right\vert^{m-1}
 - e^{r_n^{\rho_N-2\varepsilon}} \cdot \sum_{i=0}^{m-1} \left\vert f_N (z_n) \right\vert^i \\
& = e^{r_n^{\rho_N-\varepsilon}} \cdot \left\vert f_N (z_n) \right\vert^{m-1} \cdot
 \left[1 - e^{r_n^{\rho_N-2\varepsilon} - r_n^{\rho_N-\varepsilon}} \cdot \sum_{i=0}^{m-1} \left\vert f_N (z_n) \right\vert^{i- (m-1)} \right].
\end{align*}
Note that the expression inside the brackets in the last formula tends to $1$ as $n \to \infty$. Thus, it is greater than some constant
$C_1 \in (0,1)$ for $n$ large enough. A similar argument also provides
\[
e^{r_n^{\rho_N-\varepsilon}} \cdot \left\vert f_N (z_n) \right\vert^{m-1} \geq C_2 \cdot e^{m r_n^{\rho_N-\varepsilon}},
\]
for $n$ large enough and some constant $C_2 > 0$. Consequently, if $C := C_1 \, C_2$, one has for $n$ large that
\[
M \left( P(f_1,\dots,f_M)(z_n), r_n \right) \geq C \cdot e^{m r_n^{\rho_N-\varepsilon}},
\]
which implies
\begin{align*}
\rho\left(P(f_1,\dots,f_M)\right)
& = \limsup_{r \to +\infty} \frac{\log \log M\left(P(f_1,\dots,f_M),r\right)}{\log r} \\
& \geq \limsup_{n \to \infty} \frac{\log \log M\left(P(f_1,\dots,f_M),r_n\right)}{\log r_n} \\
& \geq \lim_{n \to \infty} \frac{\log \log (C \cdot e^{m r_n^{\rho_N-\varepsilon}}) }{\log r_n} \\
& \geq \rho_N - \varepsilon.
\end{align*}
Letting $\varepsilon \to 0$, the above inequalities prove
\[
\rho\left(P(f_1,\dots,f_M)\right) \geq \rho_N = \max_{k \in \mathcal{I}(P)} \rho\left(f_k\right)
\]
and the proof follows straightforwardly.
\end{proof}

From this lemma we can prove that $\mathcal{CS}_\infty\left(\mathbb{R}^m,\mathbb{C}^n\right)$ is maximal strongly algebrable, which means that the set is strongly $\mathfrak c$-algebrable.

\begin{theorem} \label{thm_RmCn}
For every \,$m \in \mathbb{N}$, the set \,$\mathcal{CS}_\infty\left(\mathbb{R}^m,\mathbb{C}^n\right)$ \,is maximal strongly algebrable in \,$\mathcal{C}\left(\mathbb{R}^m,\mathbb{C}^n\right)$.
\end{theorem}

\begin{proof}
It suffices to consider the case $n=m=1$. In fact, the case $m>1$ follows from the $m=1$ by considering the projection map from $\mathbb{R}^m$ to $\mathbb{R}$. The case $n>1$ is obtained from $n=1$ by working on each coordinate.

\vskip .1cm

For each $s>0$, select an entire function \,$\varphi_s : \mathbb{C} \to \mathbb{C}$ \,of order $s>0$.
Let $A := (0,+\infty ) \setminus \mathbb{N}$. Lemma \ref{lemagordo} assures that the set
$\left\{\varphi_s\right\}_{s \in A}$ is a system of cardinality $\mathfrak{c}$ generating a free algebra $\mathcal{A}$.

\vskip .1cm

Next, notice that any element \,$\varphi \in \mathcal{A} \setminus \{0\}$ \,may be written as a non-constant polynomial \,$P$
\,without constant term evaluated on some $\varphi_{s_1}$, $\varphi_{s_2}$, $\ldots$, $\varphi_{s_N}$:
\[
\varphi = P(\varphi_{s_1}, \varphi_{s_2}, \ldots, \varphi_{s_N}) = \sum_{|\alpha|\leq m} c_\alpha \cdot \varphi_{s_1}^{\alpha_1} \cdot \varphi_{s_2}^{\alpha_2} \cdots \varphi_{s_N}^{\alpha_N}.
\]
By Lemma \ref{lemagordo}, there exists \,$j \in \{1,\dots,N\}$ \,such that \,$\rho (\varphi ) = \rho (\varphi_{s_j}) = s_j \notin \mathbb{N}_0$.
Thus Remark \ref{rem_order} (e) guarantees that $\varphi$ is surjective. Finally, take any
\,$F \in \mathcal{CS}_\infty\left(\mathbb{R},\mathbb{C}\right)$ \,and consider the algebra
$$\mathcal{B} := \left\{\varphi\circ F\right\}_{\varphi\in\mathcal{A}}.$$
Then it is plain that
\,$\mathcal{B}$ \,is freely $\mathfrak{c}$-generated and that \,$\mathcal{B} \setminus \{0\} \subset \mathcal{CS}_\infty\left(\mathbb{R},\mathbb{C}\right)$,
as required.
\end{proof}

\section{$\sigma$-Peano spaces}

As mentioned in the introduction, the theorem of Hahn and Ma\-zur-kie-wicz provides a topological characterization of Hausdorff topological spaces that are continuous image of the unit interval $I:=[0,1]$: these are precisely the Peano spaces. In this section we investigate to\-po\-lo\-gi\-cal spaces that are continuous image of the {\it real line} and for this task the following definition seems natural.

\begin{definition} \label{def_sigPeano}
A topological space $X$ is a {\em $\sigma$-Peano space} if there exists an increasing sequence of subsets
\[
K_1 \subset K_2 \subset \dots \subset K_m \subset \dots \subset X,
\]
such that each one of them is a Peano space {\rm (}endowed with the topology inherited from $X${\rm )} and its union
amounts to the whole space, that is, $\bigcup_{n \in \mathbb{N}} K_n  = X$.
\end{definition}

From now on, CS will stand for an abbreviation of ``continuous surjective''.

\begin{proposition}\label{thm_sigmaPeano}
Let $X$ be a Hausdorff topological space. The following assertions are equivalent:
\begin{itemize}
\item[\rm (a)] $X$ is a $\sigma$-Peano space.
\item[\rm (b)] $\mathcal{CS}_\infty\left(\mathbb{R},X\right) \neq \varnothing$.
\item[\rm (c)] $\mathcal{CS}\left(\mathbb{R},X\right) \neq \varnothing$.
\end{itemize}
\end{proposition}

\begin{proof} \rm (a) $\Rightarrow$ \rm (b): Let $K_1 \subset K_2 \subset \cdots$ be an increasing sequence of Peano spaces in $X$ such that its union is the whole $X$. Fix a point $x_0 \in X$. Without loss of generality, we may suppose that $x_0 \in K_n$, for all $n\geq1$. Since Peano spaces are arcwise connected \cite[Theorem 31.2]{willard}, for each $n\geq 1$ there is a Peano map $f_n:\left[n,n+1\right]\to K_n$, that starts and ends at $x_0$, {i.e.}, $f_n(n)=x_0=f_n(n+1)$. Joining all these Peano maps with the constant path $t \in (-\infty,0] \mapsto x_0 \in K_1$, one obtains a CS map $F:\mathbb{R} \to X$.

\vskip.1cm

\noindent\rm (c) $\Rightarrow$ \rm (a): Let $f$ be a map in $\mathcal{CS}\left(\mathbb{R},X\right)$. Therefore,
\[
X  = f \left(\mathbb{R}\right) = f \left( \bigcup_{n\in \mathbb{N}}[-n,n]\right) = \bigcup_{n \in \mathbb{N}} f\left([-n,n]\right).
\]

\vskip.1cm

\noindent Since \rm (b) $\Rightarrow$ \rm (c) is obvious, the proof is done.
\end{proof}

\begin{example} {\rm (}Spaces that are $\sigma$-Peano{\rm ).} \label{exPeano}
\begin{itemize}
\item[\rm (a)] {\rm Trivially, Euclidean spaces $\mathbb{R}^n$ and Peano spaces are $\sigma$-Peano. For $1< p \leq\infty$, the Hilbert cube
\[
\mathrm{C}_p:= \prod_{n \in \mathbb{N}} \left[-\frac1n,\frac1n\right] \subset \ell_p,
\]
considered as a topological subspace of $\ell_p$,
is a compact metric space, so it is a Peano space. For each natural $k$, let $k \, \mathrm{C}_p$ be the Hilbert cube after applying an ``$k$-homogeneous dilation'' to it. Therefore, the union of Hilbert cubes $\bigcup_{k \in \mathbb{N}} k \, \mathrm{C}_p$ is a $\sigma$-Peano topological vector space, when endowed with the topology inherited from $\ell_p$.}

\item[\rm (b)] {\rm Let $\mathcal{X}$ be a separable topological vector space and $\mathcal{X}'$ be its topological dual endowed with the weak$^\ast$-topology. If $\mathcal{X}'$ is covered by an increasing sequence of (weak$^\ast$-)compact subsets, then it is $\sigma$-Peano. Indeed, when the topological dual is endowed with the weak$^\ast$-topology, its weak$^\ast$-compact subsets are metrizable (see, for instance, \cite[Theorem 3.16]{rudin}). Therefore, it will be a $\sigma$-Peano space. Clearly, this holds on the topological dual (endowed with weak$^\ast$-topology) of separable normed spaces.}
\end{itemize}

\end{example}

Recall that an {\it {\rm F}-space} is a topological vector space with complete translation-invariant metric which provides its topology.


\begin{example}{\rm (}Spaces that are not $\sigma$-Peano{\rm ).} \label{ex_notPeano}
\begin{itemize} 
\item[\rm (a)] Every $\sigma$-Peano space is se\-pa\-ra\-ble. Indeed, continuity preserves separability. In particular, $\ell_\infty$ is not $\sigma$-Peano.

\item[\rm (b)] No infinite dimensional $\mathrm{F}$-space is $\sigma$-compact (i.e., a countable union of compact spaces), and, therefore, is not $\sigma$-Peano. This is a consequence of the Baire category theorem combined with the fact that on infinite dimensional topological vector spaces, compact sets have empty interior. In particular, no infinite dimensional Banach space is $\sigma$-Peano.
\end{itemize}
\end{example}

\begin{remark}\label{rem_max_dim}
{\rm If $X$ is a $\sigma$-Peano space, then $\text{card } \mathcal{C}\left(\mathbb{R},X\right) \leq \mathfrak{c}$. Indeed, this is consequence of $\text{card } X \leq \mathfrak{c}$ (as an image of the real line), \emph{in tandem} with the fact that the separability of $\mathbb{R}$ implies that each map of $\mathcal{C}\left(\mathbb{R},X\right)$ is uniquely determined the sequence of its rational images, which defines an injective map $\mathcal{C}\left(\mathbb{R},X\right) \hookrightarrow X^\mathbb{N}$ and, therefore,
$\text{card } \mathcal{C}\left(\mathbb{R},X\right)
\leq \text{card } X^\mathbb{N}
\leq \mathfrak{c}.$}
\end{remark}


Now we state and prove the main result of this section, which provides maximal
lineability of \emph{Peano curves} on arbitrary topological vector spaces that are also $\sigma$-Peano spaces.
As in \cite{BernalOrd}, we work with some \emph{particular} Peano maps, namely, with those continuous surjections
assuming each value on an unbounded set.

\vskip .1cm

It is convenient to recall a well-known fact from set theory: a family $\{A_\lambda\}_{\lambda\in\Lambda}$ of infinite subsets of $\mathbb{N}$ is called \emph{almost disjoint} if $A_\lambda \cap A_{\lambda'}$ is finite  whenever $\lambda\neq\lambda'$. The usual procedure to generate such a family is the following (see, {e.g.}, \cite{AizPeSe}): denote by $\{q_n\}_{n\in\mathbb{N}}$
an enumeration of the rational numbers. For every irrational $\alpha$, we choose a subsequence $\{q_{n_k}\}_{k\in\mathbb{N}}$ of $\{q_n\}_{n\in\mathbb{N}}$ such that $\lim_{k\to+\infty} q_{n_k} = \alpha$ and define $A_\alpha := \{n_k\}_{k\in\mathbb{N}}$. By construction, we obtain that $\{A_\alpha\}_{\alpha \in \mathbb{R}\setminus\mathbb{Q}}$ is an almost disjoint uncountable family of subsets of $\mathbb{N}$.


\begin{theorem} \label{thm_Peano_lineab}
Let $\mathcal{X}$ be a a $\sigma$-Peano topological vector space. Then $\mathcal{CS}_\infty\left(\mathbb{R}^m,\mathcal{X}\right)$ is maximal lineable in $\mathcal{C} \left(\mathbb{R}^m,\mathcal{X}\right)$.
\end{theorem}

\begin{proof}
It is sufficient to prove the result for $m=1$. Take $g: \mathbb{N}_0 \rightarrow \mathbb{N}\times\mathbb{N}$ a bijection, and set
\[
I_{k,n} := \left[g^{-1}(k,n),g^{-1}(k,n)+1\right],
\]
for all $k,n\in\mathbb{N}$, thus $\left\{I_{k,n}\right\}_{k,n \in \mathbb{N}}$ is a family of compact intervals of $[0,+\infty)$ such that $\bigcup_{k,n\in\mathbb{N}}I_{k,n}= [0,+\infty)$, the intervals $I_{k,n}$ having pairwise disjoint interiors, and $\bigcup_{k\in\mathbb{N}} I_{k,n}$ is unbounded for every $n$. Proceeding as in the construction presented in Proposition \ref{thm_sigmaPeano}, for each $n$, we can build a CS map $f_n: \mathbb{R}\to\mathcal{X}$ with the following properties:
\begin{itemize}
\item $f_n \left(\bigcup_{k\in\mathbb{N}}I_{k,n}\right)=\mathcal{X}$;
\item for each $k\in\mathbb{N}$, on the interval $I_{k,n}\,, f_n$ starts and ends at the origin $0_\mathcal{X} \in \mathcal{X}$ and covers the $k$-th Peano subset of $\mathcal{X}$;
\item $f_n\equiv0$ on $\bigcup_{k\in\mathbb{N}}I_{k,m}$, for all $m\neq n$.
\end{itemize}
Notice that each $f_n \in \mathcal{CS}_\infty\left(\mathbb{R},\mathcal{X}\right)$, since $\bigcup_{k\in\mathbb{N}} I_{k,n}$ is unbounded.

\vskip .1cm

Now let $ \left\{J_\lambda\right\}_{\lambda\in\Lambda}$ be an uncountable family of subsets
of $\mathbb{N}$ such that each $J_\lambda$ is infinite and the set is almost disjoint. Define, for each $\lambda \in \Lambda$,
\[
F_\lambda := \sum_{n\in J_\lambda} f_n : \mathbb{R} \to \mathcal{X}.
\]
The pairwise disjointness of the interior of the intervals $I_{k,n}$ (together with the above properties of $f_n$) guarantees that $F_\lambda$ is well-defined, as well as continuous. We assert that the set
\[
\left\{F_\lambda\right\}_{\lambda\in\Lambda}
\]
provides the desired maximal lineability. The crucial point is the following argument: let $F_{\lambda_1},\dots,F_{\lambda_N}$
be distinct and $\alpha_1,\dots,\alpha_N \in \mathbb{R}$, with $\alpha_N \neq 0$. Since $J_{\lambda_N} \setminus \left(\cup_{i=1}^{N-1} J_{\lambda_i}\right)$ is infinite, we may fix $n_0 \in J_{\lambda_N} \setminus \left(\cup_{i=1}^{N-1} J_{\lambda_i}\right)$. Notice that
\[
F_{\lambda_1} = \cdots = F_{\lambda_{N-1}} \equiv 0 \ \ \text{ and } \ \ F_{\lambda_N} = f_{n_0}
\  \, \text{ on } \ \bigcup_{k\in\mathbb{N}}I_{k,n_0}.
\]
Consequently,
\[
\sum_{k=1}^N \alpha_k \cdot F_{\lambda_k} = \alpha_N \cdot f_{n_0} \ \ \text{ on } \ \bigcup_{k\in\mathbb{N}}I_{k,n_0}.
\]
Then \,$F := \sum_{k=1}^N \alpha_k \cdot F_{\lambda_k}$
is an element of $\mathcal{CS}_\infty\left(\mathbb{R},\mathcal{X}\right)$,
because the image of $\mathbb{R}$ under $F$ contains $\alpha_N \cdot f_{n_0} (\bigcup_{k\in\mathbb{N}}I_{k,n_0}) = \alpha_N \mathcal{X} = \mathcal{X}$
and each vector of ${\mathcal X}$ is the image by $f_{n_0}$ of an unbounded set.
Hence, one may easily prove that the set $\left\{F_\lambda\right\}_{\lambda\in\Lambda}$ has $\mathfrak{c}$-many linearly independent elements, and each non-zero element of its linear span also belongs to $\mathcal{CS}_\infty\left(\mathbb{R},\mathcal{X}\right)$. The maximal lineability follows from Remark \ref{rem_max_dim}.
\end{proof}

Observe that this result recovers Theorem \ref{Alb} and the second part of Theorem \ref{BerOrd}. Moreover, together with Example \ref{exPeano}, item (b), provides,

\begin{corollary}
Let $\mathcal{N}$ be a separable normed space and $\mathcal{N}'$ be its topological dual
endowed with the weak$^\ast$-topology. Then $\mathcal{CS}_\infty\left(\mathbb{R}^m,\mathcal{N}'\right)$ is maximal lineable.
\end{corollary}

Notice that this result holds in a more general framework: if $\mathcal{X}$ is a separable topological vector space and its topological dual $\mathcal{X}'$ (endowed with the weak$^\ast$-topology) is covered by an increasing sequence of (weak$^\ast$-)compact subsets, then $\mathcal{CS}_\infty\left(\mathbb{R}^m,\mathcal{X}'\right)$ is maximal lineable.

\section{Peano curves on sequence spaces} \label{sect_seq}

Throughout this section we shall deal with the space of real sequences $\mathbb{R}^\mathbb{N}$ and some of its variants. Recall that $\mathbb{R}^\mathbb{N}$ is an $\mathrm{F}$-space under the metric
\[
\mathrm{d}\left( \left(x_n\right)_n,\left(y_n\right)_n \right)
:= \sum_{n\in \mathbb{N}} \frac1{2^n} \cdot \frac{|x_n-y_n|}{1+|x_n-y_n|},
\]
and also this metric provides the product topology on $\mathbb{R}^\mathbb{N}$ (see \cite[p.175]{munkres}). From Example \ref{ex_notPeano}, it is clear that $\mathbb{R}^\mathbb{N}$ is not a $\sigma$-Peano space.


\vskip .1cm

Looking for infinite dimensional ``smaller'' subspaces of $\mathbb{R}^\mathbb{N}$ that could $\sigma$-Peano, we easily find the following example.

\begin{example}
The space $c_{00}$ of eventually null sequences {\rm (}with its natural topology induced by the sup norm{\rm )} is a $\sigma$-Peano space. Indeed, $I_n := [-n,n]^n \times \{0\}^\mathbb{N} \subset c_{00}$ defines a increasing sequence of Peano spaces in $c_{00}$, whose union results in the entire space.
\end{example}

Therefore, Theorem \ref{thm_Peano_lineab} immediately gives the following:

\begin{proposition} \label{prop_c00_lineab}
The set \,$\mathcal{CS}_\infty\left(\mathbb{R}^m,c_{00}\right)$ is maximal lineable.
\end{proposition}

It possible to provide a more ``constructive'' proof of the previous result, by just making some adjustments to an argument provided in \cite{nga}. The proof is presented below and shell be used later in order to obtain algebrability results.

\begin{proof}[2nd proof of Proposition \ref{prop_c00_lineab}]
Let $\mathbb{R}^+:=(0,+\infty)$ and $\ell_\infty^+ := \left(\mathbb{R}^+\right)^\mathbb{N} \cap \ell_\infty$. For $\mathbf{r}=\left(r_n\right)_{n\in\mathbb{N}} \in \ell_\infty^+$, let us define $\Phi_\mathbf{r} : \mathbb{R}^\mathbb{N} \to \mathbb{R}^\mathbb{N}$ by
\[
\Phi_\mathbf{r} \left(\left(t_n\right)_{n\in\mathbb{N}}\right) := \left( \phi_{r_n}(t_n) \right)_{n \in \mathbb{N}},
\]
where $\phi_{r}(t):=e^{rt}-e^{-rt}$ for each $r \in \mathbb{R}^+$.
Observe that each $\phi_{r}$ is a homeomorphism from $\mathbb{R}$ to $\mathbb{R}$ and, consequently, $\Phi_\mathbf{r}$ is a bijection. Notice that the restriction $\Phi_\mathbf{r} : \ell_\infty \to \ell_\infty$ is well defined and surjective because, for $\left(t_n\right)_n \in \ell_\infty$, one has
\[
\left\vert \phi_{r_n}(t_n) \right\vert
= \left\vert e^{r_n t_n}-e^{-r_n t_n} \right\vert
\leq e^{r_n |t_n|}
\leq e^C \,\,\, (n=1, \dots ),
\]
for some positive constant $C$. Moreover, the family $\{\phi_{r_n}\}_n$ is equicontinuous. In fact, let $t,s \in \mathbb{R}, t<s$. There exists $\zeta=\zeta(t,s,n) \in [t,s]$ such that
\[
\left\vert \phi_{r_n}(t)-\phi_{r_n}(s) \right\vert 
= \left\vert \phi_{r_n}'(\zeta) \right\vert \cdot \left\vert t-s \right\vert.
\]
But since $\mathbf{r} \in \ell_\infty^+$,
\[
|\phi_{r_n}'(u)|
= r_n\left(e^{r_n u}-e^{-r_n u}\right)
\leq r_n e^{r_n |u|}
\leq \|\mathbf{r}\|_\infty  e^{\|\mathbf{r}\|_\infty  |u|},
\]
for all real numbers $u$. Thus,
\[
\left\vert \phi_{r_n}(t)-\phi_{r_n}(s) \right\vert 
\leq \|\mathbf{r}\|_\infty e^{\|\mathbf{r}\|_\infty|t-s|} \cdot \left\vert t-s \right\vert,
\]
and, therefore, $\{\phi_{r_n}\}_n$ is equicontinuous. The continuity of the restriction
\[
\Phi_\mathbf{r} : \ell_\infty \to \ell_\infty
\]
is an immediate consequence of the equicontinuity of the coordinate maps. Since $\phi_r(0)=0$ for all $r\in\mathbb{R}^+$, we may restrict again to $\Phi_\mathbf{r} : c_{00} \to c_{00}$. Then, for a fixed map $F \in \mathcal{CS}_\infty \left(\mathbb{R},c_{00}\right)$ (see the comments after the proof of Proposition \ref{thm_sigmaPeano}), the set
\[
\left\{\Phi_\mathbf{r} \circ F\right\}_{\mathbf{r} \in \ell_\infty^+}
\]
only contains functions in $\mathcal{CS}_\infty\left(\mathbb{R},c_{00}\right)$. Working on each coordinate and using the properties of the maps $\phi_r$ as in \cite{nga}, this family provides the desired maximal lineability.
\end{proof}

The following result extends Theorem \ref{thm_RmCn} to the framework of sequence spaces.

\begin{proposition}
The set \,$\mathcal{CS}_\infty\left(\mathbb{R}^m,c_{00}\left(\mathbb{C}\right)\right)$ is maximal strongly algebrable
in \,$\mathcal{C} \left(\mathbb{R}^m,c_{00}\left(\mathbb{C}\right)\right)$.
\end{proposition}

\begin{proof}
It is sufficient to deal with the case $m=1$. The argument combines the previous constructive proof and the ideas of Theorem \ref{thm_RmCn}: let $A := (0,+\infty ) \setminus \mathbb{N}$,
and let $\varphi_s : \mathbb{C} \to \mathbb{C}$ stand for an entire function of order $s>0$ such that $\varphi_s(0)=0$.
For each $\mathbf{r}=(r_n)_n \in A^\mathbb{N}$, the map
\[
\varPhi_\mathbf{r} := \left(\varphi_{r_n}\right)_{n \in \mathbb{N}}: c_{00}\left(\mathbb{C}\right) \to c_{00}\left(\mathbb{C}\right)
\]
is well-defined, continuous and surjective. Therefore, for a fixed map
$$F \in \mathcal{CS}_\infty\left(\mathbb{R},c_{00}\left(\mathbb{C}\right)\right),$$
the set $\left\{\varPhi_\mathbf{r} \circ F\right\}_{\mathbf{r} \in A^\mathbb{N}}$
generates a free algebra, which provides the strong maximal algebrability.
\end{proof}

From Example \ref{ex_notPeano}, item (a), we know that $\ell_\infty$ is not $\sigma$-Peano. On the other hand, if we consider the product topology inherited from $\mathbb{R}^\mathbb{N}$, it is obvious that it becomes $\sigma$-Peano. In fact, $\ell_\infty = \bigcup_n [-n,n]^\mathbb{N}$. Consequently, Theorem \ref{thm_Peano_lineab} provides the maximal lineability of the set $\mathcal{CS}_\infty\left(\mathbb{R}^m,\ell_\infty\right)$ in $\mathcal{C} \left(\mathbb{R}^m,\ell_\infty\right)$.

Notice that, as we did earlier when we dealt with $c_{00}$, one could also present a constructive proof of this lineability result: for a fixed $F \in \mathcal{CS}_\infty\left(\mathbb{R},\ell_\infty\right)$ the set
\[
\left\{ \Phi_\mathbf{r}|_{\ell_\infty} \circ F \right\}_{\mathbf{r} \in \ell_\infty^+}
\]
provides the desired maximal lineability. With appropriate adaptations, a similar argument as the one employed in the proof of the algebrability of $\mathcal{CS}\left(\mathbb{R}^m,c_{00}\left(\mathbb{C}\right)\right)$ will prove that the set \,$\mathcal{CS}\left(\mathbb{R}^m,\ell_\infty\left(\mathbb{C}\right)\right)$
is maximal strongly algebrable in $\mathcal{C}\left(\mathbb{R}^m,\ell_\infty\left(\mathbb{C}\right)\right)$,
when $\ell_\infty$ is endowed with the product topology inherited from \,$\mathbb{R}^\mathbb{N}$.

\section{Peano curves on function spaces}

Let $\Lambda$ be an infinite index set. Recall that the space $\mathbb{R}^\Lambda$ of real functions $f:\Lambda\to\mathbb R$ is a  complete metric space when endowed with the metric given by
\[
\mathrm{d}\left( f,g \right)
:= \sup_{\lambda\in\Lambda} \min \left\{1, \vert f(\lambda)-g(\lambda)\vert\right\},
\]
which provides the \emph{uniform} topology on $\mathbb{R}^\Lambda$, strictly finer that the product topology (see \cite[p.~124]{munkres} for more details). Note that $\mathbb{R}^\Lambda$ is not $\sigma$-compact and, thus, cannot be $\sigma$-Peano. Indeed, suppose that $\mathbb{R}^\Lambda = \bigcup_{n \in \mathbb{N}} K_n$. We may regard $\mathbb{N}$ as a subset of $\Lambda$, and consider the standard $n$-projection $\pi_n : \mathbb{R}^\Lambda \to \mathbb{R}$, which is continuous and, so there is $x_n \in \mathbb{R} \setminus \pi_n(K_n)$. However, the function $f:\Lambda\to\mathbb R$ defined by $f(n)=x_n$, for $n \in \mathbb{N}$, and $f(\lambda)=0$, if $\lambda \notin \mathbb{N}$, does not belong to $\bigcup_n K_n = \mathbb{R}^\Lambda$.

\vskip .1cm

Let $\Lambda, \Gamma$ be infinite index sets. Clearly, if $\text{card}\,\Lambda \geq \text{card}\,\Gamma$, then $\mathcal{S} (\mathbb{R}^\Lambda , \mathbb{R}^\Gamma) \neq \varnothing$, {i.e.}, the set of surjective maps from from $\mathbb{R}^\Lambda$ onto $\mathbb{R}^\Gamma$ is non-empty. In this situation, $\Gamma$ may be seen as a subset of $\Lambda$. Keeping the notation of the proof of Proposition \ref{prop_c00_lineab}, for each $\mathbf{r}=\left(r_\gamma\right)_{\gamma \in \Gamma} \in (0,1]^\Gamma$, define $\Phi_\mathbf{r} : \mathbb{R}^\Lambda \to \mathbb{R}^\Gamma$ by $\Phi_\mathbf{r}(f)(\gamma) := \phi_{r_\gamma} \left(f(\gamma)\right)$. Since the set of coordinate maps $\left\{ \phi_\gamma := \pi_\gamma \circ \Phi_\mathbf{r} \right\}_{\gamma \in \Gamma}$ is equicontinuous, $\Phi_\mathbf{r}$  is continuous. Working with the set $\left\{ \Phi_\mathbf{r} : \mathbb{R}^\Lambda \to \mathbb{R}^\Gamma  \right\}_{\mathbf{r} \in (0,1]^\Gamma}$ and with entire maps as in Section 2, we obtain
\begin{proposition}
Let $\text{card} \,\Lambda \geq \text{card}\,\Gamma$. Then
\begin{itemize}
\item[\rm (a)] $\mathcal{CS}\left(\mathbb{R}^\Lambda,\mathbb{R}^\Gamma\right)$ \,is \,$2^{\text{card}\,\Gamma}$-lineable.

\item[\rm (b)] $\mathcal{CS}\left(\mathbb{C}^\Lambda,\mathbb{C}^\Gamma\right)$ is \,$2^{\text{card}\,\Gamma}$-algebrable.
\end{itemize}
\end{proposition}

\begin{remark}\label{gam}
{\rm There is nothing to be done in the remaining case, which is the most \emph{dramatic scenario}, namely, the case $\text{card}\,\Lambda<\text{card}\,\Gamma$. Here we have $\text{card}\,\left(\mathbb{R}^\Lambda\right) = 2^{\text{card}\,\Lambda} \leq 2^{\text{card}\,\Gamma} = \text{card}\,\left(\mathbb{R}^\Gamma\right)$. In \textbf{ZF+GCH} (Zermelo-Fraenkel + generalized continuum hypothesis) we have that $\text{card}\,\Lambda<\text{card}\,\Gamma$ implies $2^{\text{card}\,\Lambda} < 2^{\text{card}\,\Gamma}$ (although in $\textbf{ZF+MA} +\neg\textbf{CH}$ --where {\bf MA} stands for Martin Axiom-- what we have is that for all $\aleph_0\leq \mathfrak{a}, \mathfrak{b} <\mathfrak{c},\ 2^\mathfrak{a} = 2^\mathfrak{b}$). Thus, in \textbf{ZF+GCH} we get
$\mathcal{S}\left(\mathbb{R}^\Lambda,\mathbb{R}^\Gamma\right) = \varnothing,$
provided $\text{card}\,\Lambda<\text{card}\,\Gamma$, that is, there is no surjective map from $\mathbb{R}^\Lambda$ onto $\mathbb{R}^\Gamma$. Therefore, $\mathcal{CS} (\mathbb{R}^\Lambda , \mathbb{R}^\Gamma ) = \varnothing$ in this case.}
\end{remark}

\noindent \textbf{Acknowledgements.} The authors would like to thank Prof.~Dr.~J.L.~G\'amez-Merino for his fruitful comments regarding Remark \ref{gam}.


\end{document}